\documentclass[reqno,12pt]{amsart}

\usepackage[utf8]{inputenc}

\usepackage[usenames]{color}
\usepackage{enumerate}
\usepackage[T1]{fontenc}
\usepackage[utf8]{inputenc}
\usepackage{amsmath,amssymb,amsfonts,amsthm,mathtools}
\usepackage{mathrsfs}
\usepackage{enumitem}
\usepackage{xcolor}
\usepackage{url}

\numberwithin{equation}{section}

 \usepackage{amsmath,amsthm,amsfonts,amssymb,enumerate,amscd}

\newtheorem{theorem}{Theorem}[section] 
\newtheorem{corollary}[theorem]{Corollary}
\newtheorem{proposition}[theorem]{Proposition}

\newtheorem{example}[theorem]{Example}

\newtheorem{remark}[theorem]{Remark}

\newcommand{\R}{\mathbb R}

\definecolor{verdec}{rgb}{0.76,1.00,0.76}
\definecolor{lila}{rgb}{0.50,0.00,0.50}
\definecolor{rojoc}{rgb}{1.00,0.56,0.56}
\definecolor{lila1}{rgb}{1,0,1}{}
\definecolor{azulc}{rgb}{0.81,1.00,1.00}
\definecolor{lilac}{rgb}{0.78,0.56,0.79}

\def\p{\mathbf p}

\title[Functional calculus for higher-order selfadjoint operators]
{Powers and Functional Calculus for $m$-Selfadjoint and Skew $m$-Selfadjoint Operators}

\author{M. Faghih-Ahmadi, T. Berm\'udez, K. Hedayatian,  H. Zaway}
\date{\today}
\begin{document}

\maketitle

\begin{abstract}
We study powers and functional calculi of higher-order selfadjoint and skew higher-order selfadjoint operators on complex Hilbert spaces. First, we compare the behavior of their powers with the corresponding theory for $m$-isometries. We prove that if $T$ is invertible and two powers $T^r$ and $T^s$ are, respectively, $m$- and $\ell$-selfadjoint, then $T^{\gcd(r,s)}$ is $\min\{m,\ell\}$-selfadjoint. In particular, invertibility yields the expected converse for consecutive powers. We also show that two consecutive skew $m$-selfadjoint powers must be non-invertible; in fact, nilpotent. For polynomials, we obtain coefficient characterizations under a cardinality assumption on the spectrum. Finally, for holomorphic functions we give spectral characterizations according to whether the spectrum is finite, or the domain of analyticity is a region, and we derive consequences for M\"obius transformations and $m$-isometries. Several examples illustrate why the spectral hypotheses cannot be omitted.
\end{abstract}

\maketitle

\section{Introduction}

Throughout the paper, $H$ denotes a complex Hilbert space and $B(H)$ the algebra of all bounded linear operators on $H$. For $T\in B(H)$, we denote its spectrum by $\sigma(T)$ and its adjoint by $T^*$. 

An operator $T\in B(H)$ is called an $m$-selfadjoint operator if
\begin{equation}\label{def:mself}
\sum_{k=0}^{m}(-1)^{m-k}\binom{m}{k}T^{*k}T^{m-k}=0.
\end{equation}
Helton \ introduced the notion of $m$-selfadjoint operators \cite{Helton}; see also \cite{H-in}.

And  $T$ is \emph{skew $m$-selfadjoint} if $iT$ is $m$-selfadjoint. The case $m=1$ gives the usual notions of selfadjoint and skew-selfadjoint operators.

These classes could be defined by a general framework of hereditary roots of polynomials of two variables. If 
 $ 
p(x,y)=\sum_{j,k}c_{j,k}y^jx^k
 $  
is a polynomial in two variables, then $T$ is called a hereditary root of $p$ if
 $$ 
\sum_{j,k}c_{j,k}T^{*j}T^k=0.
 $$  
Thus, $m$-selfadjoint corresponds to the hereditary polynomial $(y-x)^m$, while skew $m$-selfadjoint corresponds to $(y+x)^m$. This perspective places these classes close to the theory of $m$-isometries, whose hereditary roots are $(1-yx)^m$.

McCullough and Rodman proved that every $m$-selfadjoint operator on a finite dimensional Hilbert space $H$ is of the form $A+Q$, where $A$ is a selfadjoint operator on $H$ and $Q$ is a nilpotent operator that commutes with $A$, \cite[Theorem 2.2]{mr}. However, this property is not valid on infinite Hilbert spaces; that is, there exists $m$-selfadjoint operator that is not of the form $A+Q$ where $A$ is selfadjoint and commutes with the nilpotent operator $Q$, \cite[Proposition 2.1]{sh}. In \cite[Corollary 2.4]{sh} the authors characterize $m$-selfadjoint operators $T$  such that $\sigma (T)$ is a finite set: it is proved that if $T$ is an $m$-selfadjoint operator with finite spectrum, then $T$ is of the form $A+Q$, where $A$ is selfadjoint and commutes with the nilpotent operator $Q$. 

\ \par

The purpose of this paper is to study two related questions concerning higher-order selfadjoint: to what extent can the property be recovered from powers, and how is it preserved under polynomial and holomorphic functional calculus?

The analogy with $m$-isometries is particularly useful for powers. It is known that if $T^r$ and $T^s$ are respectively an $m$-isometry and an $\ell$-isometry, then $T^t$ is an $h$-isometry, where
 $ t:=\gcd(r,s) $ and $ h:=\min\{m,\ell\} $, see \cite[Theorem 3.6]{BDM}.   
For $m$-selfadjoint operators, however, the corresponding assertion is false without an additional hypothesis. The simplest example is a non-selfadjoint and nilpotent operator $T$ with $T^2=0$. 

\ \par 

Our first main result shows that invertibility is sufficient to recover the expected conclusion for the class of higher-order selfadjoint operators. 

We then turn to polynomial functional calculus. If $T$ is $m$-selfadjoint, every polynomial with real coefficients maps $T$ to an $m$-selfadjoint operator; see \cite{sh}. Conversely, we prove that when the spectrum of $T$ contains sufficiently many points, the coefficients of a polynomial $p$ of a fixed degree for which $p(T)$ is $m$-selfadjoint are forced to be real. Also, we obtain the corresponding statement for skew $m$-selfadjoint operators. 

It is well-known that the spectrum of $m$-selfadjoint operator is contained in the real axis; hence, under suitable assumptions on the domain of analyticity, the condition that $f(\sigma(T))$ be contained in $\R$ is the condition that guarantees $m$-selfadjointness of $f(T)$. Analogous statements hold for skew $m$-selfadjoint operators, with $i\R$ in place of $\R$. We also discuss the special case of M\"obius transformations and the relation with $m$-isometries.

\ \par 

The paper is organized as follows. Section~2 recalls basic facts and establishes the result for powers. Section~3 treats general polynomials. Section~4 is devoted to holomorphic functional calculus and its applications to M\"obius transformations and $m$-isometries. We finish with the construction of  $m$-isometries using the functional calculus of $m$-selfadjoint or skew $m$-selfadjoint operators. 

\section{Powers of higher order of selfadjoint and skew selfadjoint} 

We begin with some properties that will be used.

\begin{proposition}\label{properties}
Let $T\in B(H)$. The following properties hold.
\begin{enumerate}[label=(\roman*)]
\item \cite[Corollary~24]{stankus} If $T$ is $m$-selfadjoint, then $\sigma(T)\subseteq\R$. In particular, if $T$ is skew $m$-selfadjoint, then $\sigma(T)\subseteq i\R$.

\item \cite[Corollary~2.5]{sh} If $T$ is $m$-selfadjoint and quasinilpotent, then $T$ is nilpotent of order at most $\left[ \frac{m+1}{2}\right] $, where $[s]$ denotes the integer part of $s$.

\item \cite[Theorem~2.1]{SH-E} $T$ is skew $m$-selfadjoint if and only if the Cayley transform 
 $ 
\Phi(T):=(T+I)(T-I)^{-1}
 $  
is an invertible $m$-isometry.

\item \cite[Corollary~9]{Le} If $T$ is $m$-selfadjoint, then $T^k$ is $m$-selfadjoint for every positive integer $k$.

\item \cite[Corollary~2.1]{sh} If $T$ is $m$-selfadjoint and $p$ is a polynomial with real coefficients, then $p(T)$ is $m$-selfadjoint.

\item \cite[Theorem~3.6]{BDM} If $r,s,m,\ell$ are positive integers and $T^r$ is an $m$-isometry while $T^s$ is an $\ell$-isometry, then $T^t$ is an $h$-isometry, where
 $ t:=\gcd(r,s)$ and $ h:=\min\{m,\ell\} $.
\end{enumerate}
\end{proposition}

Gu and Stankus proved that if $A, B\in B(H)$ are double commuting  (that is, $AB=BA$ and $A^*B=BA^*$), with $A$ an $n$-selfadjoint and $B$ an $\ell $-selfadjoint, then $AB$ is $(m+\ell -1)$-selfadjoint, \cite[Theorem 9]{Gu-S}. Later, Le proved that the condition about double commuting is not necessary; it is enough that $A$ and $B$ commute; see \cite[Corollary 9]{Le}. In particular, if $A$ is $m$-selfadjoint, then so is any power $A^k$. 

The commuting property given by Le is needed  in order that the product is an $m$-selfadjoint as prove the following example. 

\begin{example}
    {\rm 
     Let $S,T\in B(\mathbb C ^3) $ given by 
$$
S:=\begin{pmatrix}
    0 & i& 0\\
    -i & 0& 0\\
    0& 0& 1
\end{pmatrix} \mbox{ and } T:=\begin{pmatrix}
    1& 1& 0\\
    0&1 &0\\
    0& 0& 2 
\end{pmatrix} \;.
$$     
It is clear that $S$ is selfadjoint, and since $T$ is the sum of a selfadjoint diagonal operator and a commuting nilpotent operator of order 2, then $T$ is 3-selfadjoint, \cite[Theorem 2.2]{mr}. However, 
$$
ST:=\begin{pmatrix}
0& i& 0\\
-i & -i& 0\\
0& 0& 2  
 \end{pmatrix}  \neq \begin{pmatrix}
     -i & i & 0\\
     -i & 0& 0\\
     0& 0& 2
 \end{pmatrix} =TS\;,
$$     
and $\sigma (TS)=\sigma (ST)=\{ 2, \frac{-i+\sqrt{3}}{2}, \frac{-i-\sqrt{3}}{2} \}$. Hence neither $ST$  nor $TS$ are $m$-selfadjoint for any positive integer $m$, by  Proposition \ref{properties} (i).    
     }
\end{example}

\begin{remark}\label{nilremark}
{\rm Notice that if an operator is simultaneously skew $m$-selfadjoint and $m'$-self\-adjoint, then its spectrum is contained in $\R\cap i\R=\{0\}$. Hence it is quasinilpotent and therefore nilpotent by Proposition~\ref{properties}(ii). Conversely, if $T$ is nilpotent of order $q$, then $T$ is $(2q-1)$-selfadjoint and skew $(2q-1)$-selfadjoint.}
\end{remark}

The powers of $m$-isometries have been extensively studied \cite{BDM, Gu, Gu-S, jablonski, Le, patel}. The product of isometries is isometric. However, this property is not shared by the class of $m$-isometries. That is, it is possible to give an isometry $S$ and a 2-isometry $T$ such that neither $ST$ nor $ TS$ is an $m$-isometry for any $m$, see \cite[Example 3.1]{BMN-product}.  Proposition \ref{properties} (vi) gives sufficient conditions to obtain that some powers are $m$-isometries. 
  In particular, if $T^r$ and $T^{r+1}$ are $m$-isometries, then $T$ itself is an $m$-isometry; see \cite[Corollary 3.7]{BDM}. It is natural to ask whether an analogous statement holds for $m$-selfadjoint or skew $m$-selfadjoint operators.

For selfadjoint operators, a useful starting point is the following result.

\begin{proposition}\label{powers-basic}
Let $T\in B(H)$. Then:
\begin{enumerate}[label=(\roman*)]
\item \cite[Proposition~2.1]{DG} If $T^2$ and $T^3$ are selfadjoint and either $T$ is invertible or the range of $T^2$, $R(T^2)$, is dense in $H$, then $T$ is selfadjoint.
\item If $T^2$ and $T^3$ are selfadjoint, then $\sigma(T)\subseteq\R$ and
 $ 
T^{*2}(T^*-T)T^2=0.
 $  
\item If $T^\ell$ is skew $m$-selfadjoint and $T^k$ is skew $m'$-selfadjoint, where $k$ and $\ell$ have opposite parity, then $T$ is a nilpotent operator.
\end{enumerate}
\end{proposition}

\begin{proof}
{\it (i)} If $T$ is invertible, the result follows from the relations $T^{*2}=T^2$ and $T^{*3}=T^3=T^*T^2$. If $R(T^2)$ is dense, it is similar. 

{\it (ii)} From $T^2=T^{*2}$ and $T^3=T^{*3}$, it is clear that $\sigma (T)$ is real by the spectral mapping theorem  and Proposition~\ref{properties} (i). Also, we obtain  $T^{*2} (T^* -T)=0$ and therefore $T^{*2} (T^* -T)T^2=0$.
 
{\it (iii)} Proposition~\ref{properties}(i) gives  $ \sigma(T^k)\subseteq i\R$, and $\sigma(T^\ell)\subseteq i\R $.  
If $k$ and $\ell$ have opposite parity, the spectral mapping theorem implies $\sigma(T)=\{0\}$. Thus, $T$ is quasinilpotent, and Proposition~\ref{properties}(ii) proves that $T$ is a nilpotent operator. 
\end{proof}

The next example shows that no theorem recovering $m$-selfadjointness from two consecutive powers can hold for arbitrary non-invertible operators.

\begin{example}\label{nilcounter}
{\rm Let
 $ 
T:=\begin{pmatrix}0&1\\0&0\end{pmatrix}\in B(\mathbb {C}^2).
 $  
Then $T^2=T^3=0$, so both powers are selfadjoint, whereas $T$ is not selfadjoint. Notice that $T$ is $3$-selfadjoint. }
\end{example}

We now prove the main result of this section, which gives sufficient conditions to obtain a similar result to Proposition \ref{properties} (vi), for higher order of selfadjoint operator.

\begin{theorem}\label{powers}
Let $T\in B(H)$ and let $r,s,m,\ell$ be positive integers with  $ t:=\gcd(r,s)$ and $\qquad h:=\min\{m,\ell\}$.   
If $T$ is invertible, $T^r$ is $m$-selfadjoint and $T^s$ is $\ell$-selfadjoint, then $T^t$ is $h$-selfadjoint.
\end{theorem}

\begin{proof}
Consider the left and right multiplication operators on $B(H)$,
 $L_{T^*},$ and $R_T,$  that is,
 $ L_{T^*}(X):=T^*X$ and  $R_T(X):=XT$, for any $X\in B(H).
 $  
They commute, and for every positive integers $q$ and $n$, we have that 
$$
(L_{T^*}^q-R_{T}^q)^n(I)=
\sum_{j=0}^{n}(-1)^{n-j}\binom{n}{j}T^{*qj}T^{q(n-j)}\;.
$$
Consequently,
 $$
T^q\text{ is $n$-selfadjoint}
\quad\Longleftrightarrow\quad
(L_{T^*}^q-R_{T}^q)^n(I)=0.
 $$  

Since $T$ is invertible, then $R_T$ is invertible and $R^{-1}_T=R_{T^{-1}}$. Denote 
 $ A:=L_{T^*}R_{T}^{-1}$.   
Then $A$ commutes with $R_T$ and
 $$ 
L_{T^*}^r-R_{T}^r=R_{T}^r(A^r-I)=(A^r-I)R_{T}^r,\qquad
L_{T^*}^s-R_{T}^s=R_{T}^s(A^s-I)= (A^s-I)R^s_{T}.
 $$  
Since $T^r$ is $m$-selfadjoint and $T^s$ is $\ell $-selfadjoint, then  
 $$ 
(A^r-I)^m(I)=0,\qquad (A^s-I)^\ell(I)=0.
 $$ 
Since  $ \gcd(z^r-1,z^s-1)=z^t-1$ with $t:=gcd(r,s)$, we have
 $$ 
\gcd\big((z^r-1)^m,(z^s-1)^\ell\big)=(z^t-1)^h,
 $$  
where $h:=min\{ m, \ell \}$. 

By Bezout's identity, there exist polynomials $p,q\in\mathbb{C}[z]$ such that
 $$ 
p(z)(z^r-1)^m+q(z)(z^s-1)^\ell=(z^t-1)^h.
 $$  
Evaluating at $A$ and applying the resulting identity to $I$ gives
 $$ 
(A^t-I)^h(I)=0.
 $$  
Since
 $$ 
A^t-I=(L_{T^*}^t-R_{T}^t)R_T^{-t},
 $$  
and $L_{T^*}$, $R_{T}$, and $R^{-1}_{T}$ commute, we obtain
 $$ 
0=(A^t-I)^h(I)
=(L_{T^*}^t-R_{T}^t)^hR_{T}^{-th}(I)
=R_{T}^{-th}(L_{T^*}^t-R_{T}^t)^h(I).
 $$  
So, $ (L_{T^*}^t-R_{T}^t)^h(I)=0$, that is,  $T^t$ is $h$-selfadjoint.
\end{proof}

Invertibility is used in the proof precisely to make the right multiplication operator invertible.

\begin{corollary}
Let $T\in B(H)$ be invertible and let $r,m$ be positive integers. If both $T^r$ and $T^{r+1}$ are $m$-selfadjoint, then $T$ is $m$-selfadjoint.
\end{corollary}

The skew case behaves differently. The next result shows that there is no nontrivial analogue of Theorem~\ref{powers} for skew $m$-selfadjoint operators with two consecutive powers.

\begin{corollary}\label{powers-skew}
Let $T\in B(H)$ and let $k$ be a positive integer. If $T^k$ and $T^{k+1}$ are skew $m$-selfadjoint, then $T$ is a nilpotent operator.
\end{corollary}

\begin{proof}
By  Proposition~\ref{powers-basic} (iii), $T$ is nilpotent. 
\end{proof}

\section{Polynomials of higher order of selfadjoint and skew selfadjoint}

We next consider polynomials. The basic parity phenomenon is already visible for monomials.

An easy example, like $T:= it I$ with $t\in \mathbb R\setminus \{ 0\}$ in any Hilbert space $H$ gives that $T$ is skew selfadjoint and $T^k$ is skew selfadjoint if and only if $k$ is odd. This behavior persists in a more general setting, as it is established by the following result.

\begin{proposition}\label{monomials}
Let $T\in B(H)$ be skew $m$-selfadjoint. Then, for every positive integer $k$,
\begin{enumerate}[label=(\roman*)]
\item if $k$ is odd, then $T^k$ is skew $m$-selfadjoint;
\item if $k$ is even, then $T^k$ is $m$-selfadjoint;
\item if $k$ is even and $T^k$ is skew $m$-selfadjoint, then $T$ is nilpotent.
\end{enumerate}
\end{proposition}

\begin{proof}
Since $iT$ is $m$-selfadjoint, Proposition~\ref{properties}(iv) implies that $(iT)^k$ is $m$-selfadjoint. 

{\it (i)} If $k$ is odd, $(iT)^k=\delta  iT^k$ with $\delta \in \{ \pm 1\}$, which yields that $T^k$ is  skew $m$-selfadjoint. 

{\it (ii)} If $k$ is even, $(iT)^k=\delta  T^k$ with $\delta \in \{ \pm 1\}$, so $T^k$ is $m$-selfadjoint.

{\it (iii)} If $k$ is even and $T^k$ is  skew $m$-selfadjoint, then $T^k$ is both $m$-selfadjoint and skew $m$-selfadjoint. By Remark~\ref{nilremark}, $T^k$ is nilpotent, and hence so is $T$.
\end{proof}

\begin{corollary}
Let $T$ be a skew $m$-selfadjoint operator which is not nilpotent. Then $T^k$ is skew $m$-selfadjoint if and only if $k$ is odd.
\end{corollary}

\begin{proposition}\label{skewpoly}
Let $T\in B(H)$ be skew $m$-selfadjoint and let $p(z)=\sum_{k=0}^{n}a_kz^k$.   
\begin{enumerate}[label=(\roman*)]
\item If $a_k\in
\begin{cases}
i\R,&\text{if $k$ is even}\\
\R,&\text{if $k$ is odd}
\end{cases}$, then $p(T)$ is skew $m$-selfadjoint.
\item If $a_k\in
\begin{cases}
\R,&\text{if $k$ is even}\\
i\R,&\text{if $k$ is odd}
\end{cases}$, then $p(T)$ is $m$-selfadjoint.
\end{enumerate}
\end{proposition}

\begin{proof}
Let   $S:=iT$. Then $S$ is $m$-selfadjoint. 

{\it (i)} Define   $ q(z):=i\,p(-iz) $.   
If $k$ is even,  $a_k\in i\R$, hence $i a_k(-i)^k\in\R$. If $k$ is odd,  $a_k\in\R$, hence $i a_k(-i)^k\in\R$. Then  $q$ is a polynomial with real coefficients.  Moreover,  $ q(S)=i\,p(-iS)=i\,p(T). $  
By Proposition~\ref{properties}(v), $q(S)$ is $m$-selfadjoint. Therefore $p(T)$ is skew $m$-selfadjoint.

{\it (ii)} Define   $ q(z):=p(-iz) $.   
The coefficient of $z^k$ is $a_k(-i)^k$, which is real for every integer $k$. Thus $q$ has real coefficients and  $ q(S)=p(T) $ and by Proposition~\ref{properties}(v) gives that $p(T)$ is $m$-selfadjoint.
\end{proof}

The above proposition gives sufficient conditions for the general case. An additional hypothesis is required to establish the necessity of the coefficient conditions.

\begin{proposition}\label{Vandermonde}
Let $T\in B(H)$ be $m$-selfadjoint and let $p(z)=\sum_{k=0}^{n}a_kz^k$ be a polynomial of degree $n$. Assume that $\sigma(T)$ contains at least $n+1$ distinct points. Then
\begin{enumerate}[label=(\roman*)]
\item $p(T)$ is $m$-selfadjoint if and only if $a_k\in\R$ for every $k$.
\item $p(T)$ is skew $m$-selfadjoint if and only if $a_k\in i\R$ for every $k$.
\end{enumerate}
\end{proposition}

\begin{proof}
{\it (i)} The sufficiency  follows from Proposition~\ref{properties}(v). For necessity, choose distinct
 $ 
\lambda_0,\ldots,\lambda_n\in\sigma(T)\subseteq\R.
 $  
If $p(T)$ is $m$-selfadjoint, then  $ p(\lambda_j)\in\sigma(p(T))\subseteq\R$, for $ j=0,\ldots,n $.   
Thus
 $$ 
\begin{pmatrix}
1&\lambda_0&\cdots&\lambda_0^n\\
\vdots&\vdots&&\vdots\\
1&\lambda_n&\cdots&\lambda_n^n
\end{pmatrix}
\begin{pmatrix}a_0\\\vdots\\a_n\end{pmatrix}
=
\begin{pmatrix}p(\lambda_0)\\\vdots\\p(\lambda_n)\end{pmatrix}\in\R^{n+1}.
 $$  
Since the Vandermonde matrix is invertible because its nodes are distinct, then  $a_k\in\R$ for all $k$.

{\it (ii)} $p(T)$ is skew $m$-selfadjoint  when $-ip(T)$ is $m$-selfadjoint. Apply (i) to the polynomial $-ip$ and we get the result.
\end{proof}

\begin{proposition}\label{Vandermonde-skew}
Let $T\in B(H)$ be skew $m$-selfadjoint and let $p(z)=\sum_{k=0}^{n}a_kz^k$ be a polynomial of degree $n$. Assume that $\sigma(T)$ contains at least $n+1$ distinct points. Then
\begin{enumerate}[label=(\roman*)]
\item $p(T)$ is skew $m$-selfadjoint if and only if
 $ 
a_k\in
\begin{cases}
i\R,&\text{if $k$ is even}\\
\R,&\text{if $k$ is odd}.
\end{cases}
 $  
\item $p(T)$ is $m$-selfadjoint if and only if
 $ 
a_k\in
\begin{cases}
\R,&\text{if  $k$ is even}\\
i\R,&\text{if $k$ is  odd}.
\end{cases}
 $  
\end{enumerate}
\end{proposition}

\begin{proof}
Let  $S:=iT$. Then $S$ is $m$-selfadjoint and $\sigma(S)=i\sigma(T)$. 

{\it (i)} Define  $ q(z):=-i\,p(-iz) $.   
A direct computation gives $q(iT)=-ip(T)$ and the coefficient of $z^k$ is $a_k (-1)^{k+1}$, By Proposition~\ref{Vandermonde} applied to $q$ and $S=iT$ gives the stated coefficient conditions. 

{\it (ii)} It follows similarly, or by applying (i) to the polynomial $ip$.
\end{proof}

Note that in the above proposition, to prove the sufficient condition for $p(T)$ to be $m$-selfadjoint (or skew $m$-selfadjoint), the hypothesis on the number of $\sigma(T)$ is not needed.

The following example proves that the spectral cardinality hypothesis in Propositions~\ref{Vandermonde} and \ref{Vandermonde-skew} cannot be omitted.

\begin{example}
 {\rm Let $m\ge2$ and let $T$ be a nonzero nilpotent operator of order at most $\left[\frac{m+1}{2}\right]$. Then $T$ is $m$-selfadjoint. For the polynomial $p(z):=iz$, the operator $p(T)=iT$ is also $m$-selfadjoint, although the coefficient $i$ is not real.}
\end{example}

\section{Holomorphic functional calculus of higher-order selfadjoint and skew selfadjoint}

By  Proposition \ref{properties} (iii), we have that $T$ is skew $m$-selfadjoint if and only if $\Phi(T)$ is an invertible $m$-isometry, where $\Phi(z)=\frac{z+1}{z-1}$ is the Cayley transform. 
Indeed, we can say a similar result for the class of $ m$-isometries. 

 \begin{corollary}
     Let $T\in B(H)$. Then $T$ is an invertible $m$-isometry if and only if $\Phi(T)$ is skew $m$-selfadjoint, where $\Phi(z)=\frac{z+1}{z-1}$. 
 \end{corollary}
\begin{proof}
    It is immediate by  Proposition \ref{properties} (iii) and the fact that the inverse of $\Phi $ is equal to $\Phi $. 
\end{proof}
In the following proposition, we include some results for the functional calculus of $m$-selfadjoint.  

\begin{proposition}
Let $T\in B(H)$. 
\begin{enumerate}[label=(\roman*)]
 \item \cite[Proposition 3.2]{sh} Let $T$ be an  $m$-selfadjoint operator and $f$ be an analytic function in a neighborhood of $\sigma (T)$,  with $G$ a disk. If $f(G\cap \sigma (T))\subset \mathbb{R}$, then  $f(T)$ is an $m$-selfadjoint. 
		\item  \cite{mr} If $T$ is 3-selfadjoint, then $e^{isT} $ is 3-isometry.   
\end{enumerate}
\end{proposition}

The basic spectral principle is that, for an $m$-selfadjoint operator, the relevant real structure is encoded by the spectrum.

\begin{theorem}\label{theorem} Let $T \in B(H)$ be an $m$-selfadjoint operator (a skew $m$-selfadjoint operator) and  $f$ be an analytic function  in an open set $G$  contains $\sigma(T)$.
 Then  
 \begin{enumerate}[label=(\roman*)]
\item  If $\sigma(T)$ is a finite set then $f(T)$ is $m$-selfadjoint (skew $m$-selfadjoint) if and only if $f(\sigma(T)) \subseteq \mathbb{R}$ $ (f(\sigma(T)) \subseteq i\mathbb{R})$.
\item  If $G$ is an open and connected set of $\mathbb C$,  then $f(T)$ is $m$-selfadjoint (skew $m$-selfadjoint) if and only if $f(\sigma(T)) \subseteq \mathbb{R}$ $(f(\sigma(T)) \subseteq i\mathbb{R})$.
\item   If $G$ is an arbitrary open set, then $f(T)$ is similar to a finite direct sum of $m$-selfadjoint (skew $m$-selfadjoint) operators on a Banach space if and only if $f(\sigma(T)) \subseteq \mathbb{R}$ $(f(\sigma(T)) \subseteq i\mathbb{R})$.
\end{enumerate}
\end{theorem}
\begin{proof} First  suppose that $T$ is an $m$-selfadjoint operator.  In (i) and (ii), the direct implication follows from the spectral mapping theorem, so we prove the other side.

{\it (i)} Suppose that $T$ is quasinilpotent;  so it is nilpotent of order less than or equal to $k_0:=[\frac{m+1}{2}]$. If $f(z):= \sum_{n=0}^\infty a_n z^n$, then $a_0 \in \mathbb{R}$ and $f(T)= a_0I + S$ where $S:= \sum_{n=1}^{\infty} a_n T^n=\sum_{n=1}^{k_{0}-1} a_n T^n$ is a nilpotent operator of order less than or equal to $[\frac{m+1}{2}]$. So by  \cite[Corollary 3.1]{sh} $f(T)$ is $(2[\frac{m+1}{2}] -1)$-selfadjoint, consequently, is $m$-selfadjoint. Next, suppose that $\sigma(T)= \{x\},$ with $x \not = 0$. Again  \cite[Corollary 3.1]{sh} or \cite[Proposition 2.1]{mr},  implies that $R:= T-xI$ is $m$-selfadjoint. Since $R$ is also quasinilpotent, it is a nilpotent operator of order less than or equal to $[\frac{m+1}{2}]$. On the other hand, $g(z):= f(z+x)$ is analytic about $0$ so by the above argument $g(R)=f(T)$ is $m$-selfadjoint. Finally, suppose that $\sigma(T)= \{x_1, ..., x_n\}\subseteq \mathbb{R}$. By   \cite[Corollary 3.5]{SH-E} $T$ is the direct sum of $n$ operators as $T_j:= x_j + M_j$ where $M_j$ is a nilpotent operator of order less than or equal to $[\frac{m+1}{2}]$. An application of the previous part shows that $f(T)= \oplus_{j=1}^nf(T_j)$ is  $m$-selfadjoint. 

{\it (ii)} Suppose that $\sigma(T)$ contains an infinite subset of $\mathbb{R}$. So it has a limit point. One can assume that  $G$ contains $\sigma(T)$ is symmetric with respect to real axis, and $f$ is analytic on $G$. Now let $g(z):=\overline{f(\overline{z})}$. Then $g$ is analytic on $G$ and $\sigma(T) \subset \mathbb{R} \subseteq \{z: f(z)=g(z)\}.$ Thus the zero set of analytic function $f(z)-g(z)$ has a limit point on $G$: hence $f(z)=g(z)$ for all $z \in G$. Now,  \cite[Proposition 3.1]{sh} shows that $f(T)$ is $m$-selfadjoint. 

{\it (iii)} One implication follows from the spectral mapping theorem, that similar operators have the same spectrum, and the spectrum of finite direct sums of operators is the union of the spectrum of their factors. Now, suppose that $f(\sigma(T)) \subset  \mathbb{R}$. The compactness of $\sigma(T)$ guarantees the existance of pairewise disjoint regions $G_1,..., G_n,$ with $\sigma(T) \subseteq \cup _{j=1}^n G_j$ such that each one contains at least one element of $\sigma(T)$, and  $\cup _{j=1}^n G_j$ has connected complement. Consequently, by the Riesz decomposition theorem one can write $H= \oplus _{j=1}^n M_j$, where each $M_j$ is a closed $T$-invariant subspace. Note that this direct sum is an algebraic direct sum, which is in general, not an orthonormal direct sum. For $h= \oplus _{j=1}^n h_j \in \oplus _{j=1}^n M_j$, let $|||h||| := \sum _{j=1}^n ||h_j||$. This norm makes $ \oplus _{j=1}^n M_j$ into a Banach space. Moreover $S: (H, ||.||) \longrightarrow (\oplus _{j=1}^n M_j, |||.|||)$ defined by $Sh:= \oplus _{j=1}^n h_j$ is an invertible operator such that $ST= (\oplus _{j=1}^n T_j)S$ where $T_j= T|_{M_j}$. Note that the boundedness of $S$ follows from the boundedness of $S^{-1}$ and inverse mapping theorem. Therefore it follows that $STS^{-1}=\oplus _{j=1} ^n T_j$. On the other hand, by Runge's Theorem there is a sequence of polynomials $(p_k)_k$ such that $p_k(z) \longrightarrow f(z)$ uniformly on compact subsets of $\cup _{j=1}^n G_j$. Hence 
$$Sf(T)S^{-1} = \lim _{k \to \infty}Sp_k(T)S^{-1} =\lim _{k \to \infty} \oplus _{j=1}^n p_k(T_j) = \oplus _{j=1}^n f(T_j).$$
By the hypothesis, $T$ is $m$-selfadjoint and so is $T_j$. Consequently, an application of (ii) implies that
 $f(T_j)$ is $m$-selfadjoint for all $j$. So the proof is completed when $T$ is an $m$-selfadjoint operator. 
 
 To prove the theorem for skew $m$-selfadjoint operators replace $T$ by $-iT$ and $f(z)$ by $if(iz)$.
\end{proof}

\begin{corollary}Let $T \in B(H)$ be an $m$-selfadjoint operator (a skew $m$-selfadjoint operator) and  $f$ be an  analytic function in an open set $G$ that  contains $\sigma(T)$. Then 
 \begin{enumerate}[label=(\roman*)]
\item  If $\sigma(T)$ is a finite set then $f(T)$ is skew $m$-selfadjoint ($m$-selfadjoint) if and only if $f(\sigma(T)) \subseteq i\mathbb{R}$ $(f(\sigma(T)) \subseteq \mathbb{R})$.
\item  If $G$ is an open and connected set of $\mathbb C$,  then $f(T)$ is skew $m$-selfadjoint ($m$-selfadjoint) if and only if $f(\sigma(T)) \subseteq i\mathbb{R}$ $(f(\sigma(T)) \subseteq \mathbb{R})$.
\item   If $G$ is an arbitrary open set then $f(T)$ is similar to a finite direct sum of skew $m$-selfadjoint ($m$-selfadjoint) operators on a Banach space if and only if $f(\sigma(T)) \subseteq i\mathbb{R}$ $(f(\sigma(T)) \subseteq \mathbb{R})$.
\end{enumerate}
\end{corollary}
\begin{proof} Replace $f(z)$ by $if(z)$ when $T$ is an $m$-selfadjoint operator and when $T$ is a skew $m$-selfadjoint operator replace $T$ by $-iT$ and $f(z)$ by $f(iz)$ in the proof of the above theorem.
\end{proof}


\begin{example}
 {\rm    Let $T\in B(H)$ and let $f(z):= e^z$, $g(z):=\sin (z)$ and $ h(z):=\cos (z)$ be functions. Then 
    \begin{enumerate}
        \item If $T$ is an $m$-selfadjoint, then $f(T)$, $g(T)$ and $h(T)$ are $m$-selfadjoint. 
        \item If $T$ is a skew $m$-selfadjoint, then $if(iT)$, $ig(iT)$ and $ih(iT)$ are skew $m$-selfadjoint. 
        \item If $T$ is skew $m$-selfadjoint, then $f(iT) $ , $g(iT)$ and $h(iT)$ are $m$-selfadjoint. 
    \end{enumerate}}
\end{example}

Recall that $f(z):= \frac{az+b}{cz+d}$ is a M\"{o}bius transformation if $a, b, c, d$ are complex numbers satisfying $ad-bc \not = 0$. 

\begin{example}
{\rm Suppose that  $T$ is an $m$-selfadjoint such that $\sigma (T) \subseteq \{\alpha _1, \alpha _2\} $, where $\alpha _1\neq \alpha _2$. If  $f(z):=\frac{az+b}{cz+d}$ is a M\"{o}bius transformation analytic in an open set of $\sigma (T)$, then  by  Theorem \ref{theorem} (i),  $f(T)$ is $m$-selfadjoint (skew $m$-selfadjoint) if and only if $f(\alpha _k) \in \mathbb{R}$ ($f(\alpha _k)\in i\mathbb{R})$ for $k=1,2$. }
\end{example}

Notice that if $T$ is an $m$-selfadjoint such that the spectrum has more than two points, then the information of $f$ at two points is not enough to guarantee that $f(T)$ is an $m$-selfadjoint or a skew $m$-selfadjoint. Just take $f(z):= \frac{(1+i)z}{iz+1}$ and $T$ be an $m$-selfadjoint such that $\{ 0, 1, 2\} \subseteq \sigma (T)$.  

In the case that  $\sigma(T)$ has more than two elements, we obtain the following corollary.

\begin{corollary}Let $T \in B(H)$ be an $m$-selfadjoint operator (a skew $m$-selfadjoint operator) where $\sigma(T)$ contains more than two points and $f(z)= \frac{az+b}{cz+d}$ be a M\"{o}bius transformation that is analytic on $\sigma(T)$. Then $f(T)$ is $m$-selfadjoint (a skew $m$-selfadjoint operator) if and only if $\frac{b}{d}, \frac{a}{c}$ and $\frac{d}{c}$ are in $\mathbb{R}_{\infty}:= \mathbb{R} \cup \{\infty \}$ ($\frac{b}{d}, \frac{a}{c}$ and $\frac{d}{c}$ are in $i\mathbb{R}_{\infty}:= i\mathbb{R} \cup \{\infty \}$).
\end{corollary}
\begin{proof} Since $\sigma(T)$ contains more than two points
  by Theorem \ref{theorem}, $f(T)$ is an $m$-selfadjoint operator (a skew $m$-selfadjoint operator) if and only if $f(\mathbb{R}) \subseteq \mathbb{R}$ $(f(i\mathbb{R}) \subseteq i\mathbb{R})$ which is equivalent to $\frac{b}{d}, \frac{a}{c}$ and $\frac{d}{c}$ are in $\mathbb{R}_{\infty} \;(\frac{b}{d}, \frac{a}{c}$ and $\frac{d}{c}$ are in $i\mathbb{R}_{\infty}$).  
\end{proof}

\begin{corollary}Let $T \in B(H)$ be an $m$-selfadjoint operator (a skew $m$-selfadjoint operator) which $\sigma(T)$ contains more than two points and $f(z):= \frac{az+b}{cz+d}$ be a M\"{o}bius transformation that is analytic on $\sigma(T)$. Then $f(T)$ is skew $m$-selfadjoint ($m$-selfadjoint) if and only if $\frac{b}{d}, \frac{a}{c}$ are in $i\mathbb{R}_{\infty}$ and $\frac{d}{c}$ is in $\mathbb{R}_{\infty}$ $ (\frac{b}{d}, \frac{a}{c}$ are in $\mathbb{R}_{\infty}$ and $\frac{d}{c}$ is in $i\mathbb{R}_{\infty}$).
\end{corollary}

\begin{proof} It follows from the fact that $f(\mathbb{R}) \subseteq i\mathbb{R}$ $(f(i\mathbb{R}) \subseteq \mathbb{R}$) if and only if $\frac{b}{d}, \frac{a}{c}$ are in $i\mathbb{R}_{\infty}$ and $\frac{d}{c}$ is in $\mathbb{R}_{\infty} (\frac{b}{d}, \frac{a}{c}$ are in $\mathbb{R}_{\infty}$ and $\frac{d}{c}$ is in $i\mathbb{R}_{\infty}$).
\end{proof}

Now, the purpose is similar to  Proposition \ref{properties} (iii). That is, if we fix an $m$-selfadjoint  or skew $m$-selfadjoint operator, we obtain sufficient conditions such that $f(T)$ is an $m$-isometry for a suitable analytic function $f$  in a neighborhood of $\sigma (T)$.

\begin{theorem}\label{theorem 1} Let $T \in B(H)$ be an $m$-selfadjoint operator  or a skew $m$-selfadjoint operator.  Suppose that  $f$ is analytic in an open set $G$  containing $\sigma(T)$. Then the following holds  
 \begin{enumerate}[label=(\roman*)]
\item  If $\sigma(T)$ is a finite set, then $f(T)$ is $m$-isometry if and only if $f(\sigma(T)) \subseteq \mathbb{T}$ where $\mathbb{T}$ denotes the unit circle.
\item  If $G$ is a simply connected region, then $f(T)$ is $m$-isometry if and only if $f(\sigma(T)) \subseteq \mathbb{T}$.
\item   If $G$ is an arbitrary open set, then $f(T)$ is similar to a finite direct sum of $m$-isometric operators on a Banach space if and only if $f(\sigma(T)) \subseteq \mathbb{T}$.
\end{enumerate}
\end{theorem}

\begin{proof} The proofs of (i) and (iii) are similar to the proofs of (i) and (iii) of Theorem \ref{theorem}, respectively. 

{\it (i)} Assume that $T$ is quasinilpotent and $m$-selfadjoint. Let $f(z):= \sum _{n=0}^\infty a_n z^n$. By  \cite[Theorem 4]{Gu-S} $f(T)=a_0 I+S$ is an $m$-isometry where $|a_0|=1$ and $S$ is a nilpotent operator of order less than or equal to $[\frac{m+1}{2}]$. The other cases are similar to the proof of Theorem \ref{theorem}.  

{\it (ii)} One implication follows from the spectral mapping theorem. To prove the other side, suppose that $f(\sigma(T)) \subseteq \mathbb{T}$. One can assume that $f(z)$ is nonvanishing on $G$. Hence there is an analytic function $g: G \longrightarrow \mathbb{C}$ such that $f(z)= \exp g(z)$,  \cite[Theorem 2.2 Page 202]{C2}. Consequently,  $g(\sigma(T)) \subseteq i\mathbb{R}$ and so by part (ii) of Theorem \ref{theorem}, $g(T)$ is a skew $m$-selfadjoint operator. Thus, by   \cite[Theorem 2.1]{SH-E}, $f(T)= \exp g(T)$ is an $m$-isometry. 

{\it (iii)} Recall that $iT$ is an $m$-isometry when $T$ is.
\end{proof}

\begin{example}{\rm Suppose that $T \in B(H)$ is an $m$-selfadjoint operator and $c$ is a nonzero real number.  If $f(z) = c^{iz}$ then an application of part $(ii)$ of Theorem \ref{theorem 1} implies that $f(T)$ is an $m$-isometry. Similarly, if $S$ is a skew $m$-selfadjoint operator then $c^ S$ is an $m$-isometry. }
\end{example}

\begin{corollary}Let $T \in B(H)$ be an $m$-selfadjoint operator (a skew $m$-selfadjoint operator) which $\sigma(T)$ contains more than two points and $f(z):= \frac{az+b}{cz+d}$ be a M\"{o}bius transformation that is analytic on $\sigma(T)$. Then $f(T)$ is $m$-isometry if and only if $|a|=|c|, |b|=|d|$ and Re$(a\overline{b})=$ Re$(c \overline{d})$ $ (|a|=|c|, |b|=|d|$ and Im$(a\overline{b})=$ Im$(c \overline{d})).$
\end{corollary}
\begin{proof} It follows from the fact that $f(\mathbb{R}) \subseteq \mathbb{T}$ $(f(i\mathbb{R}) \subseteq \mathbb{T}$) if and only if $|a|=|c|, |b|=|d|$ and Re$(a\overline{b})=$ Re$(c \overline{d})$ $ (|a|=|c|, |b|=|d|$ and Im$(a\overline{b})=$ Im$(c \overline{d})).$
\end{proof}

Recall that if $G$ is an open set in $\mathbb{C}$ such that $G=\{z \in G: \overline{z} \in G\}$ and $f$ is an analytic function on $G$ then $\tilde{f}(z):= \overline{f(\overline{z})},$ is also analytic on $G$.

\begin{proposition} Let  $T\in B(H)$ be an an $m$-selfadjoint operator and $f$ is an analytic function on $\sigma(T)$ such that $
f(z) \not = 0$ for all $z\in \sigma(T)$. Then  $f(T^*)^*(f(T))^{-1}$ is an $m$-isometry.
\end{proposition}
\begin{proof}Since $\sigma(T)$ is a compact subset of the real axis, there are open discs $U_1, ...U_N$ whose centers are on the real line and 
$\sigma(T) \subseteq \cup_{i=1}^{N} U_i$
such that $\cup_{i=1}^{N} U_i$ has connected complement. Hence by Runge's Theorem \cite{C2} there is a sequence of polynomials $(p_n)_n$ such that $p_n(z) \longrightarrow f(z)$ and consequently,  $\tilde{p}_n(z) \to \tilde{f}(z)$ uniformly on compact subsets of $\cup_{i=1}^{N} U_i$. Put $p(z):=\sum_{j=0}^{k} a_jz^j$ and $r(z):=\frac{\tilde{p}(z)}{p(z)}$ when $p(z) \not =0$. Observe that 
\begin{eqnarray*}
\overline{p(z)}\left( \overline{r(z)}r(z)-1\right)p(z)&=&\tilde{p}(z)\overline{\tilde{p}(z)}-p(z)\overline{p(z)}\\
&=&\left( \sum_{j=0}^k\overline{a_j}z^j\right) \left( \sum_{j=0}^k{a_j}\overline{z}^j\right) -\left( \sum_{j=0}^k{a_j}z^j\right)\left(\sum_{j=0}^k\overline{a_j}\overline{z}^j\right)\\
&=&\sum_{\substack i \not = j, \; 0 \leq i, \; j \leq k} (\overline{a_i}(z^i \overline{z}^j- \overline{z}^iz^j)a_j) \\
&=&- \sum_{\substack  j<i, \;  0 \leq i, \; j \leq k}  (\overline{a_i}(\overline{z}^j(\overline{z}^{i-j}- z^{i-j})z^j a_j )\\
&& + \sum _{\substack  i<j, \;  0 \leq i, \; j \leq k}  (\overline{a_i}(\overline{z}^i(\overline{z}^{j-i}- z^{j-i})z^i a_j ). 
\end{eqnarray*}
Therefore, 
$$
(\overline{p(z)})^m(\overline{r(z)}r(z)-1)^m(p(z))^m= \sum_{\substack 0 \leq i \not = j \leq k }\left( p_{i,j}(\overline{z})(\overline{z}-z)^m q_{i,j}(z) \right) 
$$
for some polynomials $p_{i,j}$ and $q_{i,j}$. Now, since $\sigma(T^*)=\sigma(T) \subseteq \cup_{i=1}^NU_i$ and the function $\overline{r(z)}r(z)$ is analytic on $\cup_{i=1}^NU_i \times \cup_{i=1}^NU_i$ containing $\sigma(T) \times \sigma(T)$, by  \cite[Lemma 1.16]{A1} and substituting $T$ for $z$ and $T^*$ for $\overline{z}$, we get $r(T)= (p(T^*))^*(p(T))^{-1}$ is an $m$-isometry, for every polynomial $p$. On the other hand, since $f$ is non-vanishing on $\sigma(T)$ one can assume that $f$ is non-vanishing on $\cup_{i=1}^N U_i$ so we conclude that $r_n(z)=\frac{\tilde{p}_n(z)}{p_n(z)}$ converges uniformly to $\frac{\tilde{f}(z)}{f(z)}$ on compact subsets of this set. Thus, taking into account that the set of $m$-isometries is closed in $B(H)$, we conclude
 that $f(T^*)^*(f(T))^{-1}= \lim _{n \to \infty}p_n(T^*)^*(p_n(T))^{-1}$ is an $m$-isometry.
\end{proof}

\begin{corollary}Suppose that $S$ is a skew $m$-selfadjoint operator and $g$ is an analytic function on $\sigma(S)$ such that $
g(z) \not = 0$ for all $z\in \sigma(S)$. Then  $(g(-S^*))^*(g(S))^{-1}$ is an $m$-isometry.
\end{corollary}

\begin{proof}Put $T=iS$ and $f(z)= g(-iz)$ in the previous theorem.
\end{proof}

\ \par 

{\bf Acknowledgements}

The second author is supported by MCIN/AEI/10.13039/501100011033, Project PID2019-105011GB-I00 and by PID2022-139449NB-I00, funded by
MCIN/AEI/10.13039/501100011033/FEDER, UE.

\ \par 

\small

\smallskip
\noindent M. Faghih-Ahmadi\\
\textit{  Department of Mathematics, \\
College of Sciences, Shiraz University,\\
Shiraz 71467-13565, Iran } \\
\textit{  e-mail:} faghiha@shirazu.ac.ir

\smallskip
\noindent T. Berm\'{u}dez\\
\textit{  Departamento de An\'alisis Matem\'atico, \\ Universidad de La Laguna and Instituto de Matemáticas y Aplicaciones (IMAULL),\\ Apartado 456, 38200, 
San Cristóbal de La Laguna \\(Tenerife), Spain.}\\
\textit{  e-mail:} tbermude@ull.edu.es

\smallskip
\noindent K. Hedayatian\\
\textit{ Department of Mathematics, \\
College of Sciences, Shiraz University, \\
Shiraz 71467-13565, Iran}\\
\textit{  e-mail}: hedayati@shirazu.ac.ir

\smallskip
\noindent H. Zaway\\
\textit{ Department of Mathematics and Statistics,\\
College of Science, King Faisal University, \\
Al Ahsa, Saudi Arabia}\\
\textit{  e-mail}: hzaway@kfu.edu.sa


\begin{thebibliography}{99}

\bibitem{A1}
J. Agler,
The Arveson extension theorem and coanalytic models,
\emph{Integral Equations Operator Theory} 5 (1982), 608--631.

\bibitem{BDM}
T. Bermúdez, C. Díaz Mendoza and A. Martinón,
Powers of $m$-isometries,
\emph{Studia Math.} 208 (2012), no.~3, 249--255.

\bibitem{BMN-product} 
T. Bermúdez, A. Martinón, J. Noda,  Products of $m$-isometries. 
\emph{ Linear Algebra Appl.} 438 (2013), no. 1, 80–86.


\bibitem{C2}
J. B. Conway,
\emph{Functions of One Complex Variable II},
Springer-Verlag, New York, 1978.

\bibitem{DG}
S. Djaballah and M. Guesba,
On the class of $n$-quasi-$m$-symmetric operators,
\emph{Proc. Inst. Math. Mech. Natl. Acad. Sci. Azerb.} 50 (2024), no.~2, 200--212.


\bibitem{Gu}
C. Gu,
Functional calculus for $m$-isometries and related operators on Hilbert spaces and Banach spaces,
\emph{Acta Sci. Math. (Szeged)} 81 (2015), 605--641.

\bibitem{Gu-S}
C. Gu and M. Stankus,
Some results on higher order isometries and symmetries: products and sums with a nilpotent operator,
\emph{Linear Algebra Appl.} 469 (2015), 500--509.

\bibitem{H-in}
J. W. Helton,
Infinite dimensional Jordan operators and Sturm--Liouville conjugate point theory,
\emph{Trans. Amer. Math. Soc.} 170 (1972), 305--331.

\bibitem{Helton}
J. W. Helton, 
Operators with a representation as multiplication by $×$ on a Sobolev space. Hilbert space operators and operator algebras (Proc. Internat. Conf., Tihany, 1970), pp. 279–287. (loose errata),
\emph{Colloq. Math. Soc. János Bolyai,} Vol. 5, North-Holland, Amsterdam-London, 1972.

\bibitem{jablonski}
Z. Jabłoński,
Complete hyperexpansivity, subnormality and inverted boundedness conditions,
\emph{Integral Equations Operator Theory} 44 (2002), no.~3, 316--336.

\bibitem{Le}
T. Le,
Algebraic properties of operator roots of polynomials,
\emph{J. Math. Anal. Appl.} 421 (2015), no.~2, 1238--1246.

\bibitem{mr}
S. McCullough and L. Rodman,
Hereditary classes of operators and matrices,
\emph{Amer. Math. Monthly} 104 (1997), no.~5, 415--430.

\bibitem{patel}
S. M. Patel,
2-isometric operators,
\emph{Glas. Mat. Ser. III} 37(57) (2002), no.~1, 141--145.

\bibitem{SH-E}
M. Salehi and K. Hedayatian,
Exponentially $m$-isometric operators on Hilbert spaces,
\emph{J. Math. Anal. Appl.} 522 (2023), no.~1, Paper No. 127015, 22 pp.

\bibitem{sh}
M. Salehi and K. Hedayatian,
On higher order selfadjoint operators,
\emph{Linear Algebra Appl.} 587 (2020), 358--386.

\bibitem{stankus}
M. Stankus,
$m$-isometries, $n$-symmetries and other linear transformations which are hereditary roots,
\emph{Integral Equations Operator Theory} 75 (2013), no.~3, 301--321.

\end{thebibliography}
\end{document}